\documentclass[letterpaper, 10 pt, conference]{ieeeconf}  

\IEEEoverridecommandlockouts                              
\usepackage[compact]{titlesec}
\usepackage[wby]{callouts}
\usepackage{braket,amsfonts,xspace}
\usepackage{graphicx}
\usepackage{amsmath}
\usepackage{amssymb}
\usepackage{amsfonts}
\usepackage{array}
\usepackage{subfigure}
\usepackage{pgfplots}
\usepackage{algorithm}
\usepackage[noend]{algpseudocode}
\usepackage{graphics}
\usepackage{xcolor}
\usepackage{epsfig}
\usepackage{cite}
\usepackage{hyperref} 
\usepackage{cleveref}
\usepackage{float}
\usepackage{amsopn}
\usepackage{wrapfig}
\usepackage{bbold}

\newtheorem{theorem}{Theorem}[section]
\newtheorem{lemma}[theorem]{Lemma}

\newtheorem{example}[theorem]{Example}

\newcommand{\subscr}[2]{#1_{\textup{#2}}}

\newcommand{\until}[1]{\{1,\dots,#1\}}
\newcommand{\tvd}[2]{TVD-\uppercase{#1}$_{#2}$}

\newcommand{\real}{{\mathbb{R}}}

\newcommand{\realpositive}{{\mathbb{R}}_{>0}}
\newcommand{\realnonnegative}{{\mathbb{R}}_{\ge 0}}

\newcommand{\Ac}{\mathcal{A}}
\newcommand{\Af}{\mathsf{A}}

\newcommand{\longthmtitle}[1]{\mbox{}{\textit{(#1)}}:\xspace}

\newcommand\oprocendsymbol{\hbox{$\bullet$}}
\newcommand\oprocend{\relax\ifmmode\else\unskip\hfill\fi\oprocendsymbol}

\renewcommand{\baselinestretch}{.99}
\allowdisplaybreaks

\title{\LARGE \bf Distributed Time-Varying Coverage Control via
  Singular Perturbations}

\author{Brandon Bao \quad Jorge Cort\'es \quad Sonia Mart{\'\i}nez%
  \thanks{The authors are with the Contextual Robotics Institute at UC
    San Diego, USA (emails:bbao200905@gmail.com, cortes@ucsd.edu;
    soniamd@ucsd.edu).}  \thanks{The authors would like to thank to
    Victor Gandarillas and Melcior Pijoan for early work on the
    simulations in this manuscript. }}%

\begin{document}

\maketitle

\begin{abstract}
  This paper presents a novel dynamic coverage control algorithm
  allowing a group of robots to track an optimal-deployment
  configuration for arbitrary time-varying density functions. Building
  on singular perturbation theory, the proposed design employs a
  two-time scale separation approach, with a fast time scale
  corresponding to communication and a slow time scale corresponding
  to agent motion. The resulting algorithm is distributed over the
  2-hop Delaunay graph and, for small enough values of the
  perturbation parameter, achieves the same performance as its
  centralized counterpart.  We also introduce three discrete-time
  versions that rely only on 1-hop communication at the cost of having
  to use delayed information and formally establish their asymptotic
  convergence properties.  Our technical approach combines
  computational geometry, singular perturbation theory, generating
  functions, and linear iterations with delayed updates.  Various
  simulations illustrate the performance of the proposed algorithms.
\end{abstract}

\section{Introduction}\label{sec: intro}

Coverage control \cite{JC-SM-TK-FB:02j,FB-JC-SM:09} is a coordination
task where a multi-agent system is optimally deployed over a spatial
domain to provide for a service. A well-established approach to
achieving this task employs the so-called Lloyd's algorithm, which
drives the agents to specific spatial locations according to a static
density function. This density function is used to model e.g., the
most probable locations of targets, general external-service demands,
or human-generated commands. As practical applications involve
revising priorities when new information is revealed, a solution to a
dynamic version of this problem becomes relevant. However, finding a
distributed algorithm to dynamic coverage control becomes challenging,
as the natural extension of Lloyd's algorithm requires all-to-all
communication among the agents.  In this note, we propose a solution
to circumvent this problem via singular perturbation theory.

\emph{Background.}  Coverage control problems have received wide
attention in the literature. The work~\cite{JC-SM-TK-FB:02j}
formulates it as a locational optimization
problem~\cite{AO-BB-KS-SNC:00,ZD-HWH:01} where the objective function
is minimized via a gradient descent algorithm that leads to the
Lloyd's algorithm. The algorithm is naturally decentralized over the
so-called Delaunay graph, which is the dual of the Voronoi partition
associated with the robots' positions. By means of it, each agent
sequentially updates its region, then its position to approach its
centroid. The algorithm is robust to errors, works under asynchronous
interactions, and adapts to agent arrivals and departures. The Lloyd's
algorithm converges to the set of Centroidal Voronoi Configurations
(CVT) for static densities, which correspond to critical points of the
objective function.

The dynamic coverage problem was initially studied
in~\cite{JC-SM-TK-FB:02b} by assuming a constraint on the time-varying
density (TVD) function; namely, that the cost covered by an agent in a
CVT configuration remains constant when the density evolves in
time. Furthermore, \cite{LP-MS-QL-VK-DR-RCM-GAP:08} provided
simulations of a dynamic coverage control law that tracked Voronoi
partition centroids over time. The work~\cite{SL-YMD-ME:15} proposed a
more general and less density restrictive approach, where time-varying
coverage is achieved by following a CVT over time. This work showed
that all-to-all communication is needed to achieve perfect tracking
performance, but
nonetheless proposed a power series approximation that generates a
trade-off between tracking and decentralization. Considering more
terms of the power series results in a better approximation of the
centralized control law, at the expense of communication with agents
that are farther-hops away.  Subsequent work extended TVD methods to
evolving environments~\cite{XX-YMD:20,ET-RA-GLN:19}.

In our technical approach, we rely on singular perturbation
theory~\cite{PK-HK-JR:86,HKK:02}.  These methods have been used in
other multi-agent control problems, such as dynamic average consensus
\cite{SK-JC-SM:ecc-13}, distributed
optimization~\cite{FZ-DV-AC-GP-LS:16}, and networked dynamical
systems\cite{AC-MM:16}, but have not been employed in the context of
coverage control. We are also inspired by the analysis of
discrete-time stochastic gradient iterations with delayed updates
presented in~\cite{YA-OS-NS:20}, in particular the use of generating
functions and formal power series~\cite{HSW:94,PF-RS:09}.

\emph{Statement of Contributions.}
Our main contribution is the design of a new motion coordination
algorithm to solve dynamic coverage control problems. Notably, the
performance of the proposed algorithmic solution is close to its
centralized counterpart to an arbitrary degree of accuracy while
maintaining decentralization. To this end, we rely on singular
perturbation theory and employ a two time-scale separation approach:
in this way, a slow time-scale dynamics realizes the agents' motion,
while the fast time-scale dynamics is used to update the agents'
controls via a distributed gradient algorithm that solves a
least-squares problem.  The resulting distributed coordination
algorithm relies on communication with 2-hop neighbors in the Delaunay
graph and can be implemented in either continuous- or discrete-time.
To further keep interactions limited among agents, we propose
discrete-time versions that rely only on 1-hop communication at the
cost of having to use information delayed by one timestep, resulting
in the all-neighbors- (delayed information from both 1-hop and 2-hop
neighbors), $2\setminus \! 1$-neighbors- (delayed information from
2-hop neighbors that are not 1-hop neighbors), and
$2$-neighbors-delayed (delayed information from all 2-hop neighbors)
variants.  We formally establish their asymptotic convergence
guarantees and evaluate their performance in simulation.  The obtained
results confirm that all algorithms converge, and that those that
employ fresher information result in faster convergence.  We also run
comparisons with respect to previous baselines and implement it on a
group of differential-drive robots in the Robotarium testbed.

\section{Preliminaries}\label{sec: notation}
This section introduces basic notation and provides a quick review of
singular perturbation theory.

\emph{Voronoi Partitions and Delaunay Graph.}
We denote by $\real^d$ the $d$-dimensional real vector space and its
Euclidean 2-norm by $\|\cdot\|$.
Consider a multi-agent system
composed of $n$ agents (or sensors) with positions $p_i\in \real^d$,
for $i \in\until{n}$, in a convex set $Q \subset \real^d$. For
convenience, we denote by $p = (p_1,\dots,p_n) \in \real^{nd}$ the
vector comprising the positions of all the agents. Given
$p \in \real^{nd}$, the standard Voronoi partition of $Q$ is given by
$\mathcal{V}(p) = (V_1(p), \dots , V_n(p))$ such that, for each
$i \in \until{n}$:
\begin{equation}
  \label{def: voronoi}
  V_i (p) = \{q \in Q: \|q- p_i\|
  \leq \|q- p_j\|, \forall j \neq i\}.
\end{equation}
Each Voronoi region or cell contains $p_i \in V_i(p)$,
$i \in \until{n}$, and as a partition, it holds that
$\cup_{i=1}^n V_i(p) = Q$. Observe that $V_i \cap V_j \neq \emptyset$
only at their boundaries and, if this is the case, we say that agents
$i,j$ are Voronoi (or Delaunay) neighbors. The set of neighbors of an
agent $i$ in the Delaunay graph is denoted by $N_i$, and does not
include $i$. In the particular case that $Q$ is a polyhedron, the
Voronoi cells are polyhedra and can be computed easily. We refer the
reader to \cite{MdB-OC-MvK-MO:97,AO-BB-KS-SNC:00} for further
treatment of Voronoi diagrams.

\emph{Singular Perturbation Theory.}
We provide here a brief description of singular perturbation theory
following~\cite{HKK:02,PK-HK-JR:86}.
Consider a dynamical system of the form
\begin{subequations}\label{eq: diff eqs}
  \begin{align}
    \dot{p}          & = f(p,u,t),  \label{eq: model:a}
    \\
    \epsilon \dot{u} & = g(p,u,t), \label{eq: model:b}
  \end{align}
\end{subequations}
with $p(t_0) = p_0$ and where $0 < \epsilon \ll 1$ is a small
parameter.
This system is singularly perturbed because the parameter $\epsilon$
multiplies the derivative of $u$ in the dynamics. As a result, the
solution in the limit $\epsilon = 0$ may approximate those
of~\eqref{eq: diff eqs} for small $\epsilon$ only under some
conditions. More precisely, taking $\epsilon = 0$ in~\eqref{eq: diff
  eqs}, we obtain the reduced model:
\begin{subequations} \label{eq: reduced model}
  \begin{align}
    \dot{p} & = f(p,u,t), \label{eq: reduced model:a}   \\
    0       & = g(p,u,t). \label{eq: reduced model:b}
  \end{align}
\end{subequations}
Under the implicit function theorem, it is possible to solve for $u$
in the last algebraic equation, and obtain an expression of a root as
$u = h(p,t)$. Plugging it in~\eqref{eq: reduced model:a} yields
$\bar{p}(t)$ as the solution to
\begin{equation}\label{eq:reduced-model-with-root}
  \dot{\bar{p}} = f(\bar{p}, h(\bar{p},t), t),
\end{equation}
with $\bar{p}(0)=p_0$.  This equation is called the \textit{quasi
  steady-state model} or \textit{slow model}.
Complementary to this, define a fast time scale
$(t-t_0)/\epsilon = \eta $, and rewrite~\eqref{eq: diff eqs} as
\begin{align*}
  \frac{dp}{d\eta} & = \epsilon f(p,u, t_0+\epsilon \eta),
  \\
  \frac{du}{d\eta} & = g(p,u,t_0+\epsilon \eta). 
\end{align*}
As $\epsilon \rightarrow 0^+$ above, meaning that
$t = t_0 + \epsilon \eta \approx t_0$ and
$p = p(t_0 + \epsilon \eta, \epsilon) \approx p_0$ (i.e., the
variables $t$ and $p$ become slowly varying as $\epsilon \approx 0$),
we obtain the so-called \textit{boundary layer dynamics} given by:
\begin{subequations} \label{eq: boundary layer}
  \begin{align}
    \frac{dp}{d\eta} & = 0, \label{eq: boundary layer:a}         \\
    \frac{du}{d\eta} & = g(p,u,t_0). \label{eq: boundary layer:b}
  \end{align}
\end{subequations}
Further, by taking $u = y + h(p,t)$, and using that
$\displaystyle{\frac{\partial h}{\partial p} \frac{dp }{d\eta}} = 0$,
the boundary layer dynamics becomes:
\begin{align}\label{eq:boundarylayer2}
  \frac{dy}{d\eta} & = g(p,y+h(p,t_0),t_0). 
\end{align}
In this way, as $\epsilon \rightarrow 0^+$ the dynamics of
$\bar{p}(t)$ and $y(t/\epsilon)$, the slow model and the boundary
layer model, become decoupled.

A powerful tool to study~\eqref{eq: diff eqs} via boundary layer and
slow model dynamics is the non-smooth Tikhonov's Theorem, which we
recall next. 

\begin{theorem}[Tikhonov’s theorem~\cite{VV:97}]\label{th:Tikhonov}
  Assume that $u= h(p,t)$ is an isolated root of the
  equation~\eqref{eq: reduced model:b}, and that the functions $f$,
  $g$, and $h$ are Lipschitz, and $g$ is bounded. Assume that the
  origin is a uniformly, asymptotically stable equilibrium of the
  dynamics~\eqref{eq:reduced-model-with-root}, and that the boundary
  layer dynamics~\eqref{eq:boundarylayer2} is globally attractive for
  any fixed slow variable; see~\cite[Prop.~1]{VV:97}. Then, there is
  $\epsilon^*>0$, and compact sets $\Omega_p$ and $\Omega_u$ such
  that, for all $t_0 \ge 0$, $p_0 \in \Omega_p$,
  $y_0 = u_0-h(p_0,t_0,) \in \Omega_u$, and
  $0 < \epsilon <\epsilon^*$, there is a unique solution of the
  dynamic system \eqref{eq: diff eqs}, $p(t,\epsilon)$
  (resp.~$u(t,\epsilon)$) on $[0,\infty)$, which is uniformly
  approximated by the solution of the reduced
  model~\eqref{eq:reduced-model-with-root}, $\bar{p}(t)$, with initial
  condition $p_0$, as $\|p(t,\epsilon) - \bar{p}(t)\| = O(\epsilon)$
  (resp.~by the solution $\hat{y}(t/\epsilon)$
  to~\eqref{eq:boundarylayer2} with initial condition
  $y_0 = h(p_0,t_0,) + u_0$, as
  $\|u(t,\epsilon) - h(\bar{p}(t),t) - \hat{y}(t/\epsilon)\| =
  O(\epsilon)$), as $\epsilon \rightarrow 0^+$. Moreover, there is a
  $t_b > t_0$ and $\bar{\epsilon}^{*}\le \epsilon^*$ such that
  $\|u(t,\epsilon) - h(\bar{p}(t),t)\| = O(\epsilon)$ uniformly on
  $[t_b,\infty)$ for $\epsilon < \bar{\epsilon}^{*}$.
\end{theorem}
\smallskip

This result plays a key role in our ensuing technical analysis of the
convergence properties of decentralized time-varying coverage control.

\section{Problem Statement}\label{sec: problem def}

In this section, we define the dynamic coverage control problem that
we aim to solve.  Consider a time-varying density function
$\phi:Q \times \realnonnegative \rightarrow \realpositive$, defined
over a convex polygonal region $Q \subset \real^d$, for each
$t \in \realnonnegative$. This density function models the importance
of each point in space at each time instant. Define the coverage
control metric
\begin{equation}\label{eq: location cost}
  H(p,t) = \int_Q \min_{i=1,\dots, n} \|q-p_i\|^2 \phi(q,t)  dq.
\end{equation}
At any given time $t$, we say that $Q$ is optimally covered by the
agents located at $p(t) \in \real^{nd}$ if $p(t)$ is a minimizer
of~\eqref{eq: location cost}. It can be shown that local minima
consist of centroidal Voronoi configurations defined as follows.  Let
the partition of $Q$ be the Voronoi~\eqref{def: voronoi} tessellation
$\mathcal{V} (p) = (V_1(p),\dots, V_n(p))$. With this choice of
partition, we simplify the cost as
\begin{equation}\label{def: H}
  H(p,t) = \sum_{i = 1}^n \int_{V_i(p)}
  \|q-p_i\|^2  \phi(q,t) dq.
\end{equation}
Following~\cite{JC-SM-TK-FB:02j}, the partial derivatives of $H$
wrt $p_i$ are
\begin{equation} \label{def: dHdp}
  \frac{\partial H}{\partial p_i} = \int_{V_i(p)} -2(q-p_i)^\top \phi
  (q,t) dq.
\end{equation}
This expression can be further simplified as follows. Define the mass
and centroid of Voronoi region $V_i(p)$ as
\begin{equation}\label{def: mi, ci}
  m_i = \int_{V_i(p)} \phi (q,t)
  dq, \quad c_i = \frac{1}{m_i} \int_{V_i(p)} q \phi(q,t) dq.
\end{equation}
Then, it holds that
\begin{equation}\label{def: dHdp 2}
  \frac{\partial H}{\partial p_i} =  -2m_i(p_i-c_i)^\top,
\end{equation}
which leads to the characterization of critical points
\begin{equation}\label{eq: critical point}
  p_i(t) = c_i(p,t),\quad i \in \until{n}, \; \forall \, t \ge 0.
\end{equation}
The locations $p(t)$ and Voronoi partition for which these conditions
hold result into Centroidal Voronoi Tessellations
(CVT).
In \cite{JC-SM-TK-FB:02j,JC-SM-TK-FB:02b}, the stability and
convergence to CVT for static densities (i.e., $\phi(q,t) = \phi(q)$)
were studied via asynchronous implementations of the celebrated
Lloyd's algorithm,
\begin{equation}\label{eq: lloyd}
  \dot{p}_i (t) = - \kappa(p_i(t)-c_i(p)),\quad i \in \until{n},
\end{equation}
where $\kappa > 0$ is a proportional gain.  Note that the
dynamics~\eqref{eq: lloyd} is distributed in the sense of the Delaunay
graph~\cite{FB-JC-SM:09}.  However, for dynamic densities,~\eqref{eq:
  lloyd} does not generally guarantee convergence to CVT because it
does not account for the effect of the time-dependency of the
evolution of centroid locations.
Our aim in this paper is to synthesize a novel distributed algorithm
that guarantees $\|p_i(t) - c_i(p,t)\|\rightarrow 0$ as
$t \rightarrow \infty$, for all $i \in \until{n}$.

\section{Time-varying Coverage Control via Singular
  Perturbations}\label{sec: sing pert} 
In this section, we propose a novel coordination algorithm to solve
the time-varying coverage control problem using ideas from singular
perturbation theory.  Consider the dynamics
\begin{equation}\label{eq: TVD-C}
  \left(I - \frac{\partial c}{\partial p} \right) \dot{p} = - \kappa
  (p - c) + \frac{\partial c}{\partial t},
\end{equation}
where $c = (c_1,\dots,c_n)\in \real^{nd}$ denotes the collection of
centroids and $I$ is the $nd \times nd$ identity matrix. Note that,
under~\eqref{eq: TVD-C}, one can easily show that $\|p(t) - c(p,t)\|$
decreases exponentially fast (with rate $\kappa$) to $0$. However,
implementing~\eqref{eq: TVD-C} requires computing the inverse of the
matrix
$\left(I - \frac{\partial c}{\partial p} \right) \in \real^{nd\times
  nd}$, which is non-sparse, and this results in a a centralized
dynamic coverage control algorithm, which we denote \tvd{c}{}. We note
that the matrix is in general invertible when agents are close to
their centroids, but it can be ill-conditioned.

One way to get around the centralized computation problem is to resort
to a Neumann series approximation $(I-A)^{-1} = I + A + A^2 + \dots$
of the inverse $\big(I - \frac{\partial c}{\partial p} \big)^{-1}$,
and multiply both sides of~\eqref{eq: TVD-C} with it to obtain an
approximation. This approach, proposed in \cite{SL-YMD-ME:15}, results
in the class of algorithms denoted by \tvd{d}{k}, where $k+1$ is the
number of terms in the series expansion. The accuracy of the
approximation increases with~$k$, but so does the communication
overhead of executing the resulting coordination algorithm, as
individual agents require information from their $k$-hop neighbors.

Here, we propose an exact reformulation of~\eqref{eq: TVD-C} based on
singular perturbation theory that is amenable to distributed
computation. Consider the control $u$ that results from solving a
least-squares problem optimizing the difference between the left and
right-hand sides of \eqref{eq: TVD-C}. Formally, choose $u$ minimizing
\begin{equation}\label{eq: perturbation objective}
  F(u) = \frac{1}{2}\left\|\left( I - \frac{\partial c}{\partial
    p} \right) u - \left( -\kappa (p - c) + \frac{\partial c}{\partial
    t} \right) \right\|^2.
\end{equation}
We design a dynamics that finds this $u$ from any arbitrary initial
condition by following a gradient descent scheme. Note that 
\small
\begin{equation*}
  \nabla F(u) = \Big( I - \frac{\partial c}{\partial p} \Big)^\top
  \Big( \! \Big( I - \frac{\partial c}{\partial p} \Big) u - \Big(
  -\kappa (p - c) + \frac{\partial c}{\partial t} \Big)\! \Big).
\end{equation*}
\normalsize The component of the gradient that corresponds to agent
$i$ can be computed with information provided by 2-hop neighbors in
the Delaunay graph, and hence the resulting dynamics can be
implemented in a distributed fashion. Our basic idea is then to have
this dynamics, which only involves peer-to-peer communication run at a
much faster time timescale than the physical evolution of the
agents. Formally, we write
\begin{subequations}\label{eq: TVD-SP}
  \begin{align}
    \dot{p} & = u,
    \\
    \epsilon \dot{u} & = - \nabla F(u).
  \end{align}  
\end{subequations}
We refer to this algorithm as \tvd{sp}{\epsilon}, where
SP$_\epsilon$ comes from the singularly perturbed dynamics with
parameter $0 < \epsilon \ll 1$. The following result characterizes the
algorithm convergence properties.

\begin{theorem}\label{th: main result}
  Assume
  $I - \frac{\partial c}{\partial p} (p,t) \in \real^{nd\times nd}$ is
  invertible and its inverse is locally Lipschitz for all $t$. Then, a
  unique solution $\bar{p}(t)$ of \tvd{c}{} exists and asymptotically
  converges to $c(t)$ exponentially fast. In addition, there exists
  $\epsilon^*>0$ such that, for all $\epsilon< \epsilon^*$, the unique
  solution $p(t,\epsilon)$ to \tvd{sp}{\epsilon} approaches
  $\bar{p}(t)$ uniformly in time, i.e.,
  $\| p(t,\epsilon) - \bar{p}(t)\| = O(\epsilon)$, uniformly in
  $t\ge 0$, as $\epsilon \rightarrow 0^+$.
\end{theorem}
\begin{proof}
  We rewrite \eqref{eq: TVD-SP} as
  \begin{equation*}
    \begin{aligned}
      \dot{p} & = u,
      \\
      \dot{u} & = - \frac{1}{\epsilon} \Big( I - \frac{\partial
        c}{\partial p} \Big)^\top \Big( \! \Big( I - \frac{\partial
          c}{\partial p} \Big) u - \Big( -\kappa (p - c) +
        \frac{\partial c}{\partial t} \Big) \! \Big).
    \end{aligned}
  \end{equation*}
  From the assumptions, this dynamics is locally Lipschitz, and hence
  unique solutions of \tvd{sp}{\epsilon}, with $\epsilon >0$, exist
  from every initial condition.

  Since the matrix $M = \big(I - \frac{\partial c}{\partial p} \big)$
  is invertible, $\nabla F(u) = 0$ has a unique solution, which can be
  expressed as
  $u = h(p,t) \triangleq -M^{-1}( k(p-c) + \frac{\partial c}{\partial
    t} )$. Then, it immediately follows that the slow model,
  cf.~\eqref{eq:reduced-model-with-root}, of~\tvd{sp}{\epsilon}
  corresponds to the \tvd{c}{} dynamics~\eqref{eq: TVD-C}. As $M$ is
  invertible, there is a well defined solution $\bar{p}(t)$ to this
  dynamics from every initial condition. Define
  $e(t) = \bar{p}(t) - c(t)$, then~\eqref{eq: TVD-C} is equivalent to
  $\dot{e} = -k e$, which has a globally exponentially stable
  equilibrium point. Thus,
  $\lim_{t \rightarrow \infty} \|\bar{p}(t) - c(t)\| = 0$
  exponentially fast.  On the other hand, the boundary layer system of
  \tvd{sp}{\epsilon} is
  \begin{align*}
    \frac{d p}{d\eta}
    & = 0,
    \\
    \frac{d u}{d\eta}
    &
      = 
      -\Big( I - \frac{\partial c}{\partial p} \Big)^\top
      \Big( \!
      \Big( I - \frac{\partial c}{\partial p} \Big) u - \Big(
      -\kappa (p - c) 
      + \frac{\partial c}{\partial t} \Big) \! \Big) \nonumber
    \\
    & = 
      - \nabla F(u),
  \end{align*}
  with $d\eta = dt/\epsilon$.  For any fixed slow variable $p_0$, this
  dynamics has the unique solution of $\nabla F(u) = 0$ as a globally
  exponentially stable equilibrium. Finally, note that the functions
  $h$, as well as $f(p,u,t) = u$ and $g = -\nabla F(u)$, are locally
  Lipschitz, as the centroids of a Voronoi partition are locally
  Lipschitz, cf.~\cite{JC-SM-FB:03p}. Thus, these functions are
  globally Lipschitz and bounded over compact domains.  Collecting all
  the above arguments, we can apply Theorem~\ref{th:Tikhonov} to
  conclude the statement.
\end{proof}

The advantage of the proposed method \tvd{sp}{\epsilon} is that it is
2-hop distributed for any $\epsilon$, whereas \tvd{c}{} requires
centralized information and \tvd{d}{k} requires $k$-hop neighbor
information. The parameter $\epsilon$ has the intuitive interpretation
of the timescale of communication, so more accurately approximating
the flow \tvd{c}{} requires larger communication overhead.

\begin{example}
  We compare here the performance of Lloyd's algorithm (cf.
  eq.~\eqref{eq: lloyd}), the centralized algorithm \tvd{c}{} (cf.
  eq.~\eqref{eq: TVD-C}), the power series approximation \tvd{d}{k}
  (cf.~\cite{SL-YMD-ME:15}), and the singular perturbation algorithm
  \tvd{sp}{\epsilon} (cf. eq.~\eqref{eq: TVD-SP}).  We
  consider\footnote{Simulations are implemented on python.
  }
  10 agents performing dynamic coverage control in $Q = [-10,10]^2$
  for the following time-varying densities
  \begin{align*}
    \phi_1(q,t)
    & = e^{-\left(  \left(q_x - 2
      \sin\left(\frac{t}{5}\right)\right)^2 + \left(q_y
      \right)^2      \right),}
    \\ 
    \phi_2(q,t)
    & = e^{-\left(  \left(q_x -
      \sin\left(\frac{t}{5}\right)\right)^2 + \left(q_y +
      \sin\left(\frac{2t}{5}\right) \right)^2
      \right),}
    \\
    \phi_3(q,t)
    & = e^{-\left(  \left(q_x - 2
      \cos\left(\frac{t}{5}\right)\right)^2 + \left(q_y -
      2 \sin\left(\frac{t}{5}\right)\right)^2  \right),}
  \end{align*}
  representing different moving targets. As a measure of performance,
  we use the integral $ \int_0^T H(p(t),t) dt$ of the coverage
  cost~\eqref{def: H} over the duration of the simulation. We set
  $\kappa =1$, $T=31.5$s, and $dt=0.1$s.

  \begin{table}[htb]
    \centering
    \caption{Comparison of total costs for 10 agents.}\label{tab: 10
      agents}
    \begin{tabular}{|c|c|c|c|}
      \hline
      Algorithm      & $\phi_1$ & $\phi_2$ & $\phi_3$ \\
      \hline
      Lloyd          & 52.11    & 50.57    & 62.82    \\
      \hline
      \tvd{d}{0}     & 49.10    & 48.55    & 55.41    \\
      \hline
      \tvd{d}{1}     & 44.39    & 44.23    & 47.90    \\
      \hline
      \tvd{d}{2}     & 42.68    & 42.70    & 44.75    \\
      \hline
      \tvd{d}{3}     & 41.95    & 41.97    & 43.23    \\
      \hline
      \tvd{sp}{0.1}  & 46.26    & 46.26    & 53.29    \\
      \hline
      \tvd{sp}{0.05} & 42.49    & 43.07    & 42.93    \\
      \hline
      \tvd{sp}{0.01} & 41.40    & 41.45    & 41.48    \\
      \hline
      \tvd{c}{}      & 41.84    & 42.47    & 41.46    \\
      \hline
    \end{tabular}
  \end{table}
  
  We simulate the Neumann series approximation algorithm \tvd{d}{k}
  for $k \in \{ 0,1,\dots,3\}$ and the proposed singular perturbation
  algorithm \tvd{sp}{\epsilon} with
  $\epsilon \in \{ 0.1, 0.05, 0.01\}$.  Table~\ref{tab: 10 agents}
  presents the results.  One can see that as $\epsilon$ gets smaller
  (i.e., more rounds of communication take place to decide the control
  action for agent motion), the performance of \tvd{sp}{\epsilon}
  improves.  \tvd{d}{k} and \tvd{sp}{\epsilon} both perform better
  than Lloyd's algorithm, and both \tvd{sp}{0.01} and \tvd{d}{3}
  (requiring 3-hop neighbor information) achieve costs comparable to
  the cost of \tvd{c}{}, with \tvd{sp}{0.01} performing slightly
  better. \oprocend
\end{example}

\section{1-Hop Distributed Discrete-Time Implementations with
  Delayed Information}\label{sec: num implementation}

In this section, we study discretizations of the dynamics
\tvd{sp}{\epsilon} that require communication only with 1-hop
neighbors at the cost of using partially delayed information.

\subsection{2-Hop Distributed Discretization of \tvd{sp}{\epsilon}}

For convenience, we use the notation $u^k = u(k t)$, for a
discretization step $\delta t$. A simple forward Euler integration in
$p$ gives
\begin{subequations}
\begin{equation}\label{eq:discrete-p}
  p^{k+1} = p^{k} + u^{k}  \delta t, 
\end{equation}  
for each $k \ge 0$.  The control $u^k$ is given by a discretization of
the boundary layer dynamics $du/d\eta = - \nabla F(u)$ of~\eqref{eq:
  TVD-SP}, where $ p \equiv p^{k}$ is held constant.
Since
$\eta$ is a much faster time scale than $t$, for each update of $t$,
we need to implement $N \approx 1/ \epsilon$ updates in $\eta$. We
assume that $\epsilon$ is chosen so that $N$ is an integer value, and
index the fast updates in $u$ as $u_\ell = u(\ell \eta)$. In this way,
$u_{N k } \equiv u^k$, for each $k$. The undelayed, discretized
gradient-descent dynamics of the boundary layer equation is simply
\begin{equation}\label{eq:discrete-u}
  u_{\ell+1} = u_{\ell} - \delta \eta \nabla F(u_{\ell}), \qquad \ell \ge 0.
\end{equation}
\end{subequations}
We next investigate the information exchange among agents required to
execute this algorithm. 

For convenience, we write $\nabla F(u)$ as
\begin{subequations}
  \label{eq:gradient-F}
  \begin{align}
    \nabla F(u) &
           =   A u + b,
  \end{align}
  where
  \begin{align}
    A & = \Big( I - \frac{\partial c}{\partial p} \Big)^\top \Big( I -
        \frac{\partial c}{\partial p} \Big) \in \real^{nd \times nd},
    \\
    b & = -\Big( I - \frac{\partial c}{\partial p}
        \Big)^\top\Big(-\kappa (p - c) + \frac{\partial c}{\partial
        t}\Big) \in \real^{nd} .
  \end{align}
\end{subequations}
The following result provides an explicit description of the entries
of $A$ and~$b$.

\begin{lemma}\label{lem:block-entries}
  The block entries of $A$ and $b$ in~\eqref{eq:gradient-F} are, for
  $i,j \in \until{n}$,
  \begin{align*}
    A_{ii}
    & = \Big(I-\frac{\partial c_i}{\partial
      p_i}\Big)^\top \!  \Big(I-\frac{\partial c_i}{\partial
      p_i}\Big) + \sum_{j\in N_i} \frac{\partial
      c_j}{\partial p_i}^\top \! \frac{\partial
      c_j}{\partial p_i}  \in \real^{d \times d},
    \\
    A_{ij}
    & = - \Big( \! \Big(
      I-\frac{\partial c_i}{\partial p_i}\Big)^\top \frac{\partial
      c_i}{\partial p_j} + \frac{\partial c_j}{\partial p_i}^\top
      \Big(I-\frac{\partial c_j}{\partial p_j} \Big) \Big) \,
      \mathbb{1}_{N_i}(j)
    \\
    & \qquad + 
      \sum_{k \in
      N_i \cap N_j} \frac{\partial c_k}{\partial p_i}^\top
      \frac{\partial c_k}{\partial p_j} \, \mathbb{1}_{N_i^2}(j) \in
      \real^{d \times d}, 
    \\    
    b_i
    & = -\Big(I-\frac{\partial c_i}{\partial p_i}\Big)^\top
      \Big(-\kappa(p_i-c_i)  +\frac{\partial c_i}{\partial
      t} \Big)
    \\
    & \qquad +
      \sum_{j\in N_i}\frac{\partial c_j}{\partial
      p_i}^\top \Big(-\kappa(p_j-c_j)+\frac{\partial c_j}{\partial
      t}\Big) \in \real^{d},
  \end{align*}
  where $\mathbb{1}_{S} (\cdot)$ denotes the indicator function for
  the set~$S$.
\end{lemma}
\begin{proof} 
  The component corresponding to agent~$i$ is
  \begin{equation*}
    \begin{aligned}
      \nabla_i
      & F(u) = \Big( \! \Big(I-\frac{\partial c_i}{\partial
        p_i}\Big)^\top \!  \Big(I-\frac{\partial c_i}{\partial
        p_i}\Big) + \sum_{j\in N_i}\frac{\partial
        c_j}{\partial p_i}^\top \! \frac{\partial
        c_j}{\partial p_i} \! \Big) u_i
      \\
      & - \sum_{j \in N_i}
        \Big( \! \Big(I-\frac{\partial c_i}{\partial
        p_i}\Big)^\top \frac{\partial c_i}{\partial p_j} + \frac{\partial
        c_j}{\partial p_i}^\top \Big(I-\frac{\partial c_j}{\partial p_j}
        \Big) \! \Big) u_j
      \\
      &
        + \sum_{j\in N_i} \Big(\sum_{k \in N_j}
        \frac{\partial c_j}{\partial p_i}^\top \frac{\partial c_j}{\partial
        p_k} u_k \Big)+ b_i .
    \end{aligned}
  \end{equation*}
  By interchanging the indices $j$ and $k$ in the last term, and
  describing all agents such that $k\in N_i$ and $j \in N_k$ as
  $j \in N_i^2$ and $k \in N_i \cap N_j$ instead (here, $N_i^2$ denotes
  the 2-hop Delaunay neighbors of $i$), the equation can be rewritten
  as
  \begin{equation*}
    \begin{aligned}
      \nabla_i
      & F(u) = \Big( \! \Big(I-\frac{\partial c_i}{\partial
        p_i}\Big)^\top \! \Big(I-\frac{\partial c_i}{\partial
        p_i}\Big) + \sum_{j\in N_i} \frac{\partial
        c_j}{\partial p_i}^\top \frac{\partial
        c_j}{\partial p_i} \Big) u_i
      \\
      &
        - \sum_{j \in
        N_i} \Big( \! \Big(I-\frac{\partial c_i}{\partial
        p_i}\Big)^\top \frac{\partial c_i}{\partial p_j} + \frac{\partial
        c_j}{\partial p_i}^\top \Big(I-\frac{\partial c_j}{\partial p_j}
        \Big) \Big)  u_j
      \\
      &
        + \sum_{j \in N_i^2} \Big( \sum_{k \in
        N_i \cap N_j} \frac{\partial c_k}{\partial p_i}^\top
        \frac{\partial c_k}{\partial p_j} \Big) u_j + b_i.
    \end{aligned}
  \end{equation*}
  The expression for $A_{ii}$ and $A_{ij}$ then readily
  follows. Finally, the expression for $b_i$ follows from its
  definition in~\eqref{eq:gradient-F}.
\end{proof}

The result in Lemma~\ref{lem:block-entries} is useful to understand
the intricacies of the information exchange required to implement the
discretization of \tvd{sp}{\epsilon}. With the notation of this
result,  we rewrite~\eqref{eq:discrete-u} as
\begin{align}\label{eq:2-hop-fresh-info}
  u_{\ell+1} = u_\ell - \delta \eta (A u_{\ell} + b), \qquad \ell \ge
  1.
\end{align}
First, note that the matrix $A$ is symmetric. Second, the computation
of the blocks of the matrix $A$ requires 2-hop position neighbor
information.  Since during the $N$ steps of the discretized boundary
layer dynamics~\eqref{eq:discrete-u} that correspond to the $1$ step
of the slow dynamics~\eqref{eq:discrete-p}, the agents' position is
held constant, the transmission of \emph{agent position information}
only needs to occur once.

\begin{theorem}\label{th:convergence}
  Assume $\nabla F(u_*) = 0$ has a unique solution~$u_*$ and
  $A \succ 0$. Then, for $\delta \eta \in (0,1/(2\subscr{\lambda}{max}(A)))$,
  the dynamics~\eqref{eq:2-hop-fresh-info} satisfy $u_\ell \rightarrow u_*$
  as $\ell \rightarrow \infty$.
\end{theorem}
\smallskip

Note that, since the update~\eqref{eq:2-hop-fresh-info} requires 2-hop
\emph{control neighbor information} at every step, each agent cannot
execute it if only communication with 1-hop neighbors is possible at
every step. Instead, to address this, agents could explore the use
delayed information about its 2-hop neighbors: essentially, this
corresponds to the information that its 1-hop neighbor received the
previous step, which is relayed at the current step. This is what we
discuss in the following.

\subsection{1-Hop Distributed Implementation with Delayed Updates}

The discussion above justifies why in the following we investigate
partially-delayed gradients versions of~\eqref{eq:discrete-u}.  Set
initially $u_{i,0} = u_{i,1} =0$, for all $i$. Assume that, at each
round $\ell \ge 2$, agent $i$ transmits information about its current
$u_{i,\ell}$ as well as relays the values $u_{j,\ell-1}$ received from
its 1-hop neighbors at the previous round $\ell-1$. Then, each agent
receives information from its 2-hop neighbors with 1 time-step
delay. With this information exchange, one can write the
\emph{all-neighbors-delayed} update as,
\begin{equation}\label{eq:homogeneous}
  u_{\ell+1} = u_\ell - \delta \eta (A u_{\ell -1} + b), \qquad \ell \ge
  1.
\end{equation}
In this case, all agents keep in memory their previous values
$u_{i,\ell-1}$ and update the gradient with a delay of 1 time step
across all agents (even though they have more up to date information
about their 1-hop neighbors and themselves). The next result
characterizes the convergence properties of~\eqref{eq:homogeneous}.

\begin{theorem}\label{th: delayed discretized}
  Assume $\nabla F(u_*) = 0$ has a unique solution~$u_*$ and
  $A \succ 0$. Then, for
  $\delta \eta \in (0,1/(4\subscr{\lambda}{max}(A)))$, the
  dynamics~\eqref{eq:homogeneous} satisfy $u_\ell \rightarrow u_*$ as
  $\ell \rightarrow \infty$.
\end{theorem}
\begin{proof}
  The result follows by rewriting~\eqref{eq:homogeneous} as
  \begin{align*}
    \begin{bmatrix}
      u_{\ell+1}
      \\
      u_{\ell}
    \end{bmatrix}
    =
    \begin{bmatrix}
      I  & - \delta \eta A
      \\
      I & 0
    \end{bmatrix}
    \begin{bmatrix}
      u_{\ell}
      \\
      u_{\ell-1}
    \end{bmatrix}
    +
    \begin{bmatrix}
      - \delta \eta b
      \\
      0
    \end{bmatrix}
  \end{align*}
  and noting that the eigenvalues of the system matrix are of the form
  $\frac{1\pm\sqrt{1-4 \delta \eta \mu }}{2}$, with $\mu$ an eigenvalue of~$A$.
\end{proof}

According to~\cite[Lemma~1]{YMD-SL-ME:15}, if the locational cost is
locally strongly convex at the CVT, then there exists a constant such
that, if the distance between the agents and their CVT is below it,
then eig$(\frac{\partial c}{\partial p}) < 1$ (recall from the proof
of Theorem~\ref{th: main result} that, if
$ \Big(I - \frac{\partial c}{\partial p} \Big)$ is invertible, then
$\nabla F(u) = 0$ has a unique solution).  In our numerical
simulations, eigenvalues of $A$ are in the range $(0, 1]$
when agents are near their time-varying centroids.

The update~\eqref{eq:homogeneous} uses only delayed information to
evaluate the term $A u$ in $\nabla F (u)$. However, this can be
refined to employ the up-to-date information that each agent actually
has about itself and its 1-hop neighbors. To formalize these, we next
describe equivalent forms of writing $\nabla F(u)$.

\begin{lemma}\label{lem:gradient-decomp}
  Define the matrices\footnote{The superindex $f$ stands for `fresh'
    and $d$ stands for `delayed'.} $\Ac^f$ and $\Ac^d$ block-wise as
  follows: for each $i \in \until{n}$, let \small
  \begin{alignat*}{2}
    \Ac^f_{ii}
    &= A_{ii} ,
    \\
    \Ac^f_{ij}
    & = - \Big( \!  \Big(I-\frac{\partial c_i}{\partial
      p_i}\Big)^\top \! \frac{\partial c_i}{\partial p_j} + \frac{\partial
      c_j}{\partial p_i}^\top \! \Big(I-\frac{\partial c_j}{\partial p_j}
      \Big)\!\Big), && \quad \text{if } j \in N_i\setminus N_i^2,
    \\
    \Ac^f_{ij}
    & = - \Big( \!  \Big(I-\frac{\partial c_i}{\partial
      p_i}\Big)^\top \! \frac{\partial c_i}{\partial p_j} + \frac{\partial
      c_j}{\partial p_i}^\top \! \Big(I-\frac{\partial c_j}{\partial p_j}
      \Big)\!\Big)
    \\
    & \qquad + \sum_{k \in N_i \cap N_j}\frac{\partial c_k}{\partial
      p_i}^\top \! \frac{\partial c_k}{\partial p_j}, && \quad \text{if } j
                                                 \in N_i \cap N_i^2,
    \\
    \Ac^f_{ij}
    & = 0, && \quad \text{otherwise},
    \\
    \Ac^d_{ij}
    & = \sum_{k \in N_i \cap N_j} \frac{\partial
      c_k}{\partial p_i}^\top \! \frac{\partial c_k}{\partial p_j}, && \quad
                                                               \text{if } j \in N_i^2 \setminus N_i,
    \\
    \Ac^d_{ij} & = 0 , && \quad \text{otherwise}.
  \end{alignat*}
  \normalsize
  Then, $\nabla F(u)$ can be expressed as
  $ \nabla F(u) = \Ac^f u + \Ac^du + b$.
  
  Also, define the matrices $\Af^f$ and $\Af^d$ block-wise as
  follows: for each $i \in \until{n}$, let
  \small
  \begin{alignat*}{2}
    \Af^f_{ii}
    & = A_{ii} ,
    \\
    \Af^f_{ij}
    & = - \Big( \! \Big(I-\frac{\partial c_i}{\partial
      p_i}\Big)^\top \! \frac{\partial c_i}{\partial p_j} + \frac{\partial
      c_j}{\partial p_i}^\top \! \Big(I-\frac{\partial c_j}{\partial p_j}
      \Big)\!\Big), && \quad \text{if } j \in N_i,
    \\
    \Af^f_{ij} & = 0, && \quad \text{otherwise},
    \\
    \Af^d_{ij} & = \sum_{k \in
                 N_i \cap N_j} \frac{\partial c_k}{\partial p_i}^\top \!
                 \frac{\partial c_k}{\partial p_j},
               && \quad \text{if } j \in N_i^2 ,
    \\
    \Af^d_{ij} & = 0, &&  \quad \text{otherwise}.
  \end{alignat*}
  \normalsize
  Then, $\nabla F(u)$ can be expressed as
  $ \nabla F(u) = \Af^f u + \Af^d u + b$.
\end{lemma}
\smallskip

Note that the matrices $\Ac^f$, $\Ac^d$, $\Af^f$, and $\Af^d$ are all
symmetric. Motivated by the decompositions in
Lemma~\ref{lem:gradient-decomp}, we consider the generic update
\begin{equation}
  \label{eq:generic}
  u_{\ell+1} = u_\ell - \delta \eta (\overline{A}^f u_\ell + \overline{A}^d u_{\ell-1} +
  b), \quad \ell \ge 2,
\end{equation}
with $u_{0} = u_1 = 0$.
Note that $u_*$ is still a fixed point of~\eqref{eq:generic} as long
as $\overline{A}^f + \overline{A}^d = A$.

According to the first decomposition of the gradient $\nabla F$ in
Lemma~\ref{lem:gradient-decomp}, we choose $\overline{A}^f = \Ac^f$
and $\overline{A}^d = \Ac^d$ in~\eqref{eq:generic} to obtain the
\emph{$2\setminus \! 1$-neighbors-delayed} update
\begin{equation}\label{eq:discrete-u2}
  u_{\ell+1} = u_\ell - \delta \eta (\Ac^f u_\ell + \Ac^d u_{\ell-1} +
  b), \quad \ell \ge 2,
\end{equation}
with $u_{0} = u_1 = 0$.  According to this update, delayed information
is only used for those 2-hop neighbors that are not 1-hop neighbors
too.  According to the second decomposition of the gradient $\nabla F$
in Lemma~\ref{lem:gradient-decomp}, we choose $\overline{A}^f = \Af^f$
and $\overline{A}^d = \Af^d$ in~\eqref{eq:generic} to obtain the
\emph{$2$-neighbors-delayed} update
\begin{equation}\label{eq:discrete-u3}
  u_{\ell+1} = u_\ell - \delta \eta (\Af^f u_\ell + \Af^d u_{\ell-1} +
  b), \quad \ell \ge 2,
\end{equation}
with $u_0 = u_1 = 0$.  This update does not distinguish among 2-hop
neighbors that are/are not also 1-hop neighbors, and delays all the
information of the 2-hop neighbors.
Observe that both dynamics use more up-to-date neighbor information
than~\eqref{eq:homogeneous}, with dynamics~\eqref{eq:discrete-u2}
being the one that takes full advantage of it.  The following result
characterizes the convergence properties of updates of the
form~\eqref{eq:generic} (and hence, of~\eqref{eq:discrete-u2}
and~\eqref{eq:discrete-u3}) employing the theory of generating
functions, cf.~\cite{HSW:94,PF-RS:09}.

\begin{theorem}\label{th: hybrid discretized}
  Assume $(\overline{A}^f + \overline{A}^d) u + b=0$ has a unique
  solution~$u_*$, $\overline{A}^{f,\top} = \overline{A}^f \succ 0$,
  and $\overline{A}^{d,\top}=\overline{A}^d$.  Then, if
  $\subscr{\lambda}{min}(\overline{A}^f) >
  \max\{\subscr{\lambda}{max}(\overline{A}^d),|\subscr{\lambda}{min}(\overline{A}^d)|\}$,
  there is a sufficiently small $\delta \eta$ such that the
  dynamics~\eqref{eq:generic} satisfy $u_\ell \rightarrow u_*$, as
  $\ell \rightarrow \infty$.
  \end{theorem}
\begin{proof}
  The proof proceeds by showing that
  $e_\ell = u_\ell - u^* \rightarrow 0$, as $\ell \rightarrow
  \infty$. The error dynamics is given by
  $e_{\ell+1} = e_\ell - \delta \eta (\overline{A}^f e_\ell+ \overline{A}^d
  e_{\ell-1})$, for $\ell \in \mathbb{N}$, starting at some
  $e_1, e_0 \in \real^{nd}$.
    Denote by $B = I-\delta \eta \overline{A}^f$, and
  $\Lambda = \delta \eta \overline{A}^d$, and define 
  the formal power series expansion
  \begin{align*}
    f(z) & = \sum_{\ell = 0}^{+\infty}  e_\ell z^\ell \in \real[\![z]\!]^{nd}, 
  \end{align*}
  where $z $ is the formal series expansion variable, and
  $\real[\![z]\!]$ is the ring of formal power series expansions with
  coefficients in $\real$.\footnote{The above can be understood as a
    formal series expansion vector the entries of which are given by
    formal series expansions in real coefficients; that is, if
    $e = (e_i) \in \real^{nd}$, then
    $\sum_{\ell = 0}^{+\infty} e_\ell z^\ell = (\sum_{l = 0}^{+\infty}
    e_{i,\ell} z^\ell )$. Induced operations of this formal series
    expansion vector with a matrix formal power series expansion allow
    us to carry out the operations that follow.} In the following, we
  will aim to obtain bounds of $e_\ell$ so that
  $\|e_\ell\|\rightarrow 0$ as $\ell \rightarrow \infty$.  Note that
  it holds that
    \begin{align*}
      f(z) & = e_0 +  e_1 z + \sum_{\ell=2}^\infty (B e_{\ell-1} - \Lambda
      e_{\ell-2}) z^\ell  \\ & = e_0 +  e_1 z +
      B z \sum_{\ell=1}^\infty
      e_{\ell}z^{\ell}   - \Lambda \sum_{\ell=2}^\infty e_{\ell-2} z^\ell\\ & = e_0 +
      e_1 z +  [B z f(z)  - B e_0 z] -\Lambda z^2 f(z).
    \end{align*}
    From here, we obtain:
    \begin{align*}
      (I - B z + \Lambda  z^2) f(z) & = e_0 + (e_1 - B e_0)z.
    \end{align*}
    Define $\pi(z) \triangleq I - B z + \Lambda z^2$, which is an
    invertible formal series expansion in the ring
    $\real^{nd\times nd}[\![z]\!]$ with inverse $\pi^{-1}(z)$ (or
    $\frac{1}{\pi(z)}$ in the sequel.\footnote{The formal series
      expansion $\pi(z)$ is invertible in the ring
      $\real^{nd \times nd} [\![z]\!]$ since $I$ is invertible. The
      coefficients of such an inverse $\mu(z)$ can be found from the
      identity $\pi(z)\mu(z) = 1$, and by solving recursively for
      the coefficients of $\mu(z)$.}) Therefore,
    \begin{align*}
      f(z) & = \pi^{-1}(z)(e_0 + (e_1 - B e_0)z).
    \end{align*}
    
    Denote the coefficients of a formal series
    $S(z) = \sum_{\ell = 0}^\infty a_\ell z^\ell$ as
    $[z^\ell] S(z) \triangleq a_\ell$.  
    One can show that
  \begin{align*}
    [z^\ell] f(z) \triangleq e_\ell = [z^\ell]
    (\pi^{-1}(z) (e_0 + (e_1 - B e_0)z)), \quad \ell \ge 2,
  \end{align*}
  Thus, the coefficient of $z^\ell$ of $f(z)$, $[z^\ell] f(z)$, is
  given by:
  \begin{align*}
    [z^\ell] f(z) & = e_\ell = [z^\ell] (\pi^{-1}(z) (e_0 + (e_1 - B e_0)z)) \\
    & = [z^\ell] \pi^{-1}(z) e_0 + [z^{\ell -1}]  \pi^{-1}(z) (e_1 - B e_0). 
  \end{align*}

  To calculate this expression, we first find
  $[z^\ell](\pi^{-1}(z))$. Under the assumptions, there is a
  sufficiently small $\delta \eta$ such that
  $B^2 - 2\Lambda = (B^2 - 2\Lambda)^\top >0$; cf.
  Lemma~\ref{le:roots-matrix} in the
  Appendix.
  Therefore, the matrix roots of $\pi(z) = I - B z + \Lambda z^2 =0$,
  must satisfy
      \begin{align*}
        \Lambda M_\pm & = \frac{B \pm \sqrt{B^2 - 2\Lambda}}{2}. 
    \end{align*}
    Note that a root $M_\pm$ (also denoted individually as $M_+$,
    $M_-$) must be invertible since $I - BM_\pm + \Lambda M_\pm^2 = 0$
    implies $I = (B - \Lambda M_\pm)M_\pm$, therefore
    $\det(M_\pm) \neq 0$. It also holds that
    $\pi'(M_\pm) = -B +2 \Lambda M_{\pm} = \pm \sqrt{B^2 - 2\Lambda}$
    is invertible, for a sufficiently small $\delta \eta$. Observe that the
    inverse of the root matrices, $M_\pm^{-1}$, must also be roots of
    the matrix polynomial $p(z) = \pi(z^{-1})z = z^2 - Bz +\Lambda$.
        Therefore, a partial fraction
    expansion of $\pi^{-1}(z)$ can be obtained in terms of these
    invertible matrices as follows:
    {\small 
      \begin{align*}
        &\frac{1}{\pi(z)} = \frac{\pi'(M_+)^{-1}}{(z - M_+)} +
        \frac{\pi'(M_-)^{-1}}{(z - M_-)} \\
        &= -\frac{1}{\pi'(M_+)M_+ (1 - M_+^{-1} z)} - 
        \frac{1}{\pi'(M_-)M_- (1 - M_-^{-1} z)} \\
        & = 
        -\pi'(M_+)^{-1}M_+^{-1} \sum_{\ell=0}^\infty
        M_+^{-\ell} z^\ell  -\pi'(M_-)^{-1}M_-^{-1} \sum_{\ell=0}^\infty
        M_-^{-\ell} z^\ell,
      \end{align*}
    } where we have employed the geometric series expression
    $1/(1-a) = \sum_{\ell = 0}^\infty a^{\ell}$.  
    Plugging this  into $[z^\ell]f(z)$ above, 
  \begin{align*}
    &  [z^\ell]f(z) = [z^\ell]\left(\frac{1}{\pi(z)}e_0 \right)+ [z^{\ell-1}]
      \left( \frac{1}{\pi(z)} (e_1 - B e_0)\right)
    \\
    & =
      -(\pi'(M_+)^{-1}M_+^{-1}
      M_+^{-\ell} + \pi'(M_-)^{-1}M_-^{-1}
      M_-^{-\ell}) e_0
    \\
    &\qquad -(\pi'(M_+)^{-1}M_+^{-1} 
      M_+^{-\ell+1} +
    \\
    & \qquad \qquad +\pi'(M_-)^{-1}M_-^{-1} 
      M_-^{-\ell+1}) (-e_1 + B e_0) , \;  \ell\ge 2.
  \end{align*}

  To prove convergence, we need a bit more work to bound the
  eigenvalues of each $M_\pm^{-1}$. 
 First, note that 
    \begin{align*}
      M_\pm^{-1} & = \frac{1}{2}\left( B \pm \sqrt{B^2 - 2\Lambda}\right).
    \end{align*}
    thus, $(M_\pm^{-1})^\top = M_\pm^{-1}$ have real eigenvalues. Any of these roots, 
    $R$, satisfies
    \begin{equation*}
      p(R) = R^2 - BR + \Lambda =0.
    \end{equation*}
    Therefore, for each eigenvector $x$ of $R$ such that $x^\top x = 1$, it holds
    that
    \begin{equation*}
      x^\top R^2 x - x^\top BR x + x^\top \Lambda x = 0. 
    \end{equation*}
    which implies 
    \begin{equation*}
      \lambda^2 - \lambda x^\top B x + x^\top \Lambda x = 0,
    \end{equation*}
    for the associated eigenvalue.  From Lagrange's
    theorem~\cite{PB-MM-DS:17}, the absolute values of the roots of this
    polynomial will be bounded by the sum of the coefficients
    $|x^\top (I-\delta \eta \overline{A}^f)x| + |x^\top \delta \eta \overline{A}^d
    x|$.  Further, from Courant-Fischer's theorem, and for $\delta \eta$
    sufficiently small, we have
    \begin{align*}
      & |x^\top (I-\delta \eta \overline{A}^f)x| + |x^\top \delta \eta
        \overline{A}^d x| \le 
        1 - \delta \eta \subscr{\lambda}{min}(\overline{A}^f)
      \\&
      \qquad + \delta \eta
      \max\{\subscr{\lambda}{max}(\overline{A}^d),|\subscr{\lambda}{min}(\overline{A}^d)|\}
      \triangleq \gamma. 
    \end{align*}
    Therefore, the previous inequality is less than one if
    $\subscr{\lambda}{min}(\overline{A}^f) >
    \max\{\subscr{\lambda}{max}(\overline{A}^d),|\subscr{\lambda}{min}(\overline{A}^d)|\}
    $.

  From here, one can take the norm of each coefficient, $e_\ell$ and
  upper-bound it in terms of the eigenvalues of the matrix roots of
  $\pi(z)$, which are related to those of $\overline{A}^f$ and
  $\overline{A}^d$. Putting $M_1 \equiv M_+$,
     and $M_2 \equiv M_-$, we obtain
{\small
    \begin{align*}
      & \| e_\ell \| \le
      \\
      &\| \sum_{i=1}^2 \left(-
        \pi'(M_i)^{-1} M_i^{-(\ell+1)}e_0 +  \pi'(M_i)^{-1} M_i^{-\ell} (e_1 - B
        e_0)\right)\|
      \\
      & \le
        \sum_{i=1}^2\left( \| e_0 \| \|\pi'(M_i)^{-1}\| \|
        M_i^{-1}\|^{\ell+1} + \dots \right.
      \\ & \left. \qquad \qquad
           \qquad \dots + \| e_1 - B e_0 \|
           \|\pi'(M_i)^{-1}\| \| M_i^{-1}\|^{\ell} \right)
      \\ & \qquad = O(
           \max |\lambda(M_+^{-1})|^\ell) \le O(\gamma^\ell).
    \end{align*}}
\end{proof}

\begin{table}[htb]
  \centering
  \caption{1-hop distributed algorithms with delayed
    updates.}\label{tab:delayed-algos}
  \begin{tabular}{|c|c|c|}
    \hline
    Update
    & Equation
    & Description of delay
    \\
    \hline
    all-neighbors-delayed
    & \eqref{eq:homogeneous}
    & \parbox[t]{.4\linewidth}{Control
      information from
      both 1-hop and
      2-hop neighbors
      delayed by 1 timestep}
    \\
      \hline
    $2\setminus \! 1$-neighbors-delayed
    & \eqref{eq:discrete-u2}
    & \parbox[t]{.4\linewidth}{Control
      information from
      2-hop neighbors that are not  1-hop too
      delayed by 1 timestep}
    \\
    \hline
    $2$-neighbors-delayed
    & \eqref{eq:discrete-u3}
    & \parbox[t]{.4\linewidth}{Control
      information from all
      2-hop  neighbors
      delayed by 1 timestep}
    \\
    \hline      
  \end{tabular}
\end{table}

\begin{example}\longthmtitle{Numerical simulation of 1-hop distributed
    algorithms with delayed updates}
  We illustrate here the performance of the various 1-hop distributed
  algorithms with delayed updates discussed in this section,
  summarized for convenience in Table~\ref{tab:delayed-algos}.  We
  again consider 10 agents, this time deployed over a set
  $Q \subseteq \real^3$ with 20 by 20 by 20 units and with density
  $ \phi_4(q,t) = e^{-\left( \left(q_x - 4
        \cos\left(\frac{t}{5}\right)\right)^2 + \left(q_y - 4
        \sin\left(\frac{t}{5}\right)\right)^2 + \left(q_z\right)^2
    \right)}$.
  We set $\epsilon= 0.05$, $\kappa=1$, $T=31.5$s, and $dt=0.1$s.
  Figure~\ref{fig:10-agents-delayed} shows the result of the
  simulations.  Consistent with the use of the control information of
  each algorithm, cf. Table~\ref{tab:delayed-algos}, we observe that
  the coordination algorithms with the
  $2\setminus \! 1$-neighbors-delayed update converges faster than the
  $2$-neighbors-delayed update, which in turn converges faster than
  the all-neighbors-delayed one. This is also verified by the total
  costs for each run.
  \oprocend
\end{example}

\begin{figure*}[htb]
  \centering
  \subfigure[]{\includegraphics[width = 0.28
    \linewidth]{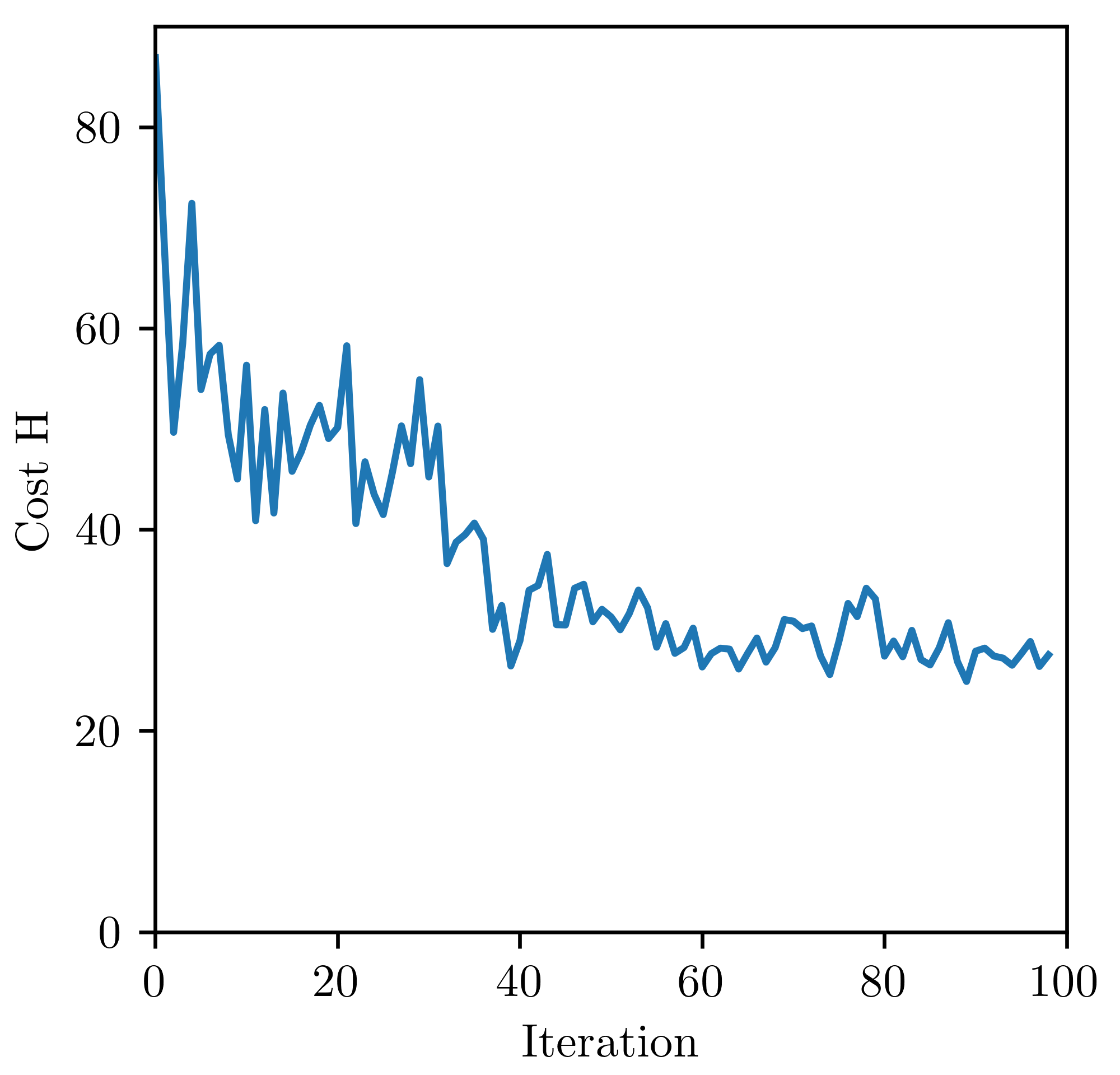}}
  \subfigure[]{\includegraphics[width = 0.28
    \linewidth]{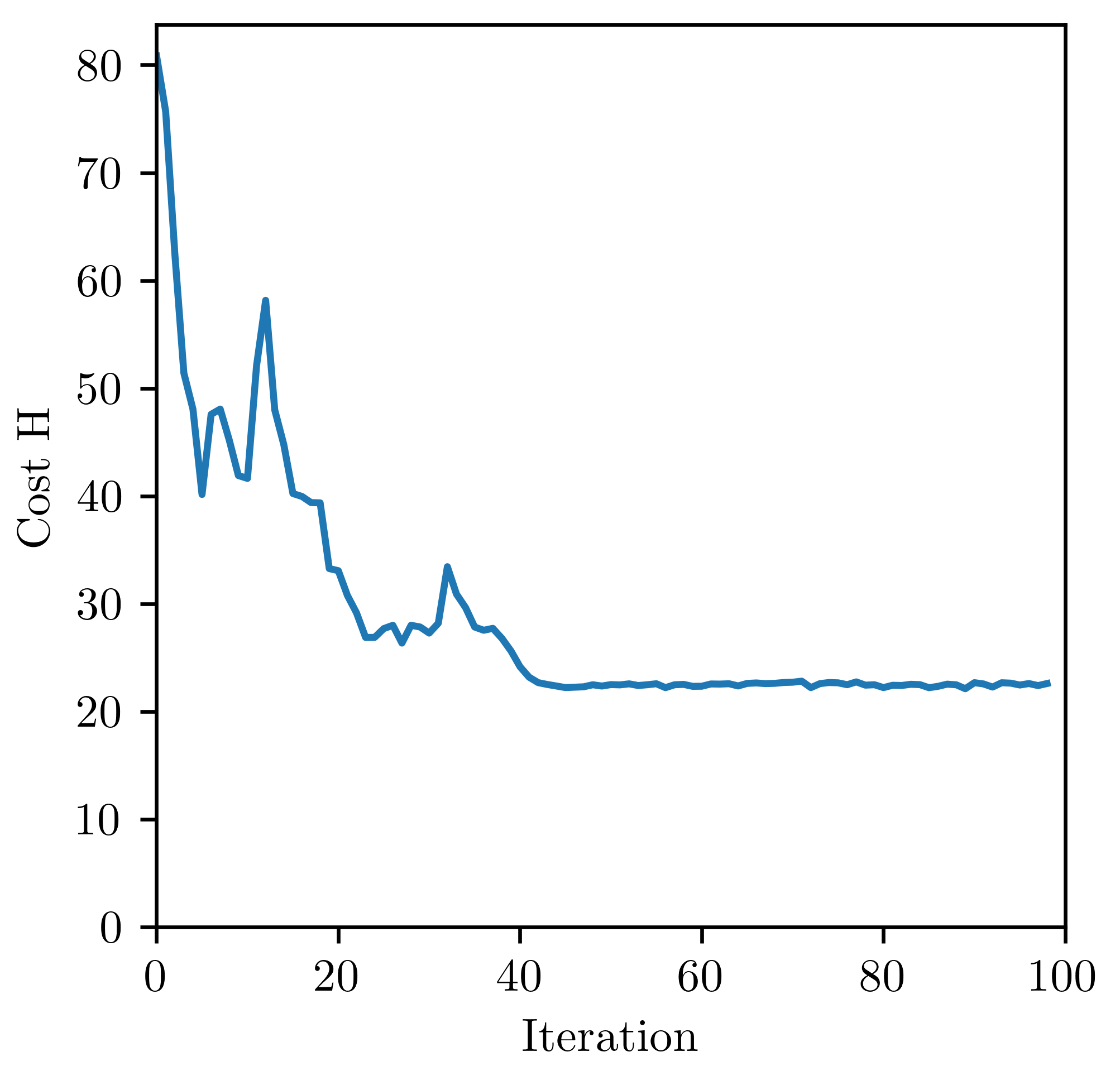}}
  \subfigure[]{\includegraphics[width = 0.28
    \linewidth]{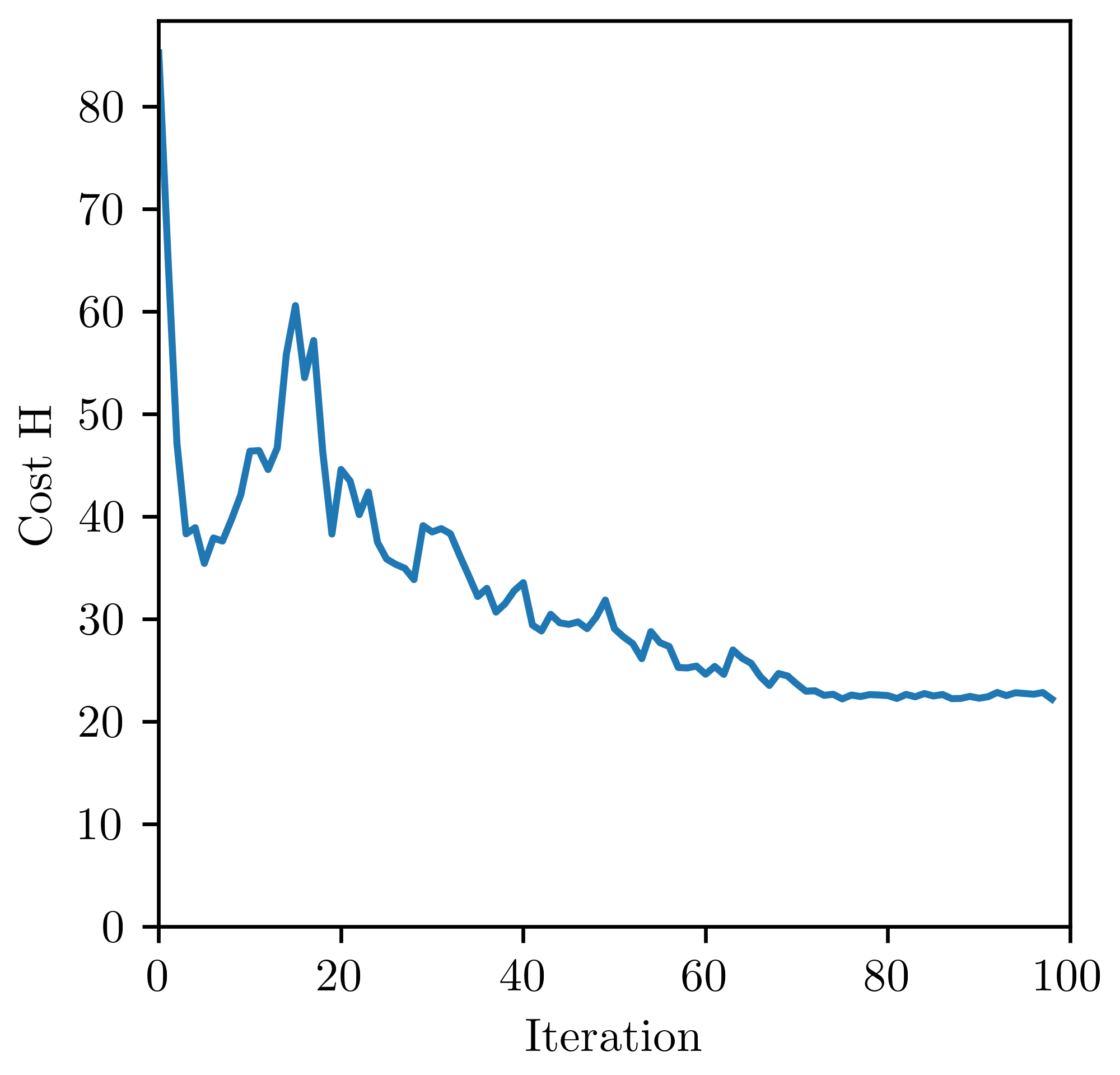}}
  \caption{Evolution of the coverage control metric~\eqref{eq:
      location cost} with $n=10$ agents and dynamic density
    function~$\phi_4$ under the (a) all-neighbors-delayed
    update~\eqref{eq:homogeneous}, with total cost 3687.5 (b)
    $2\setminus \! 1$-neighbors-delayed update~\eqref{eq:discrete-u2},
    with total cost 2876.11, and (c) $2$-neighbors-delayed
    update~\eqref{eq:discrete-u3}, with total cost
    3175.49.}\label{fig:10-agents-delayed}
  \vspace*{-1ex}
\end{figure*}

\begin{figure}[htb]
  \centering
  \subfigure[]{\includegraphics[width=.475\linewidth]{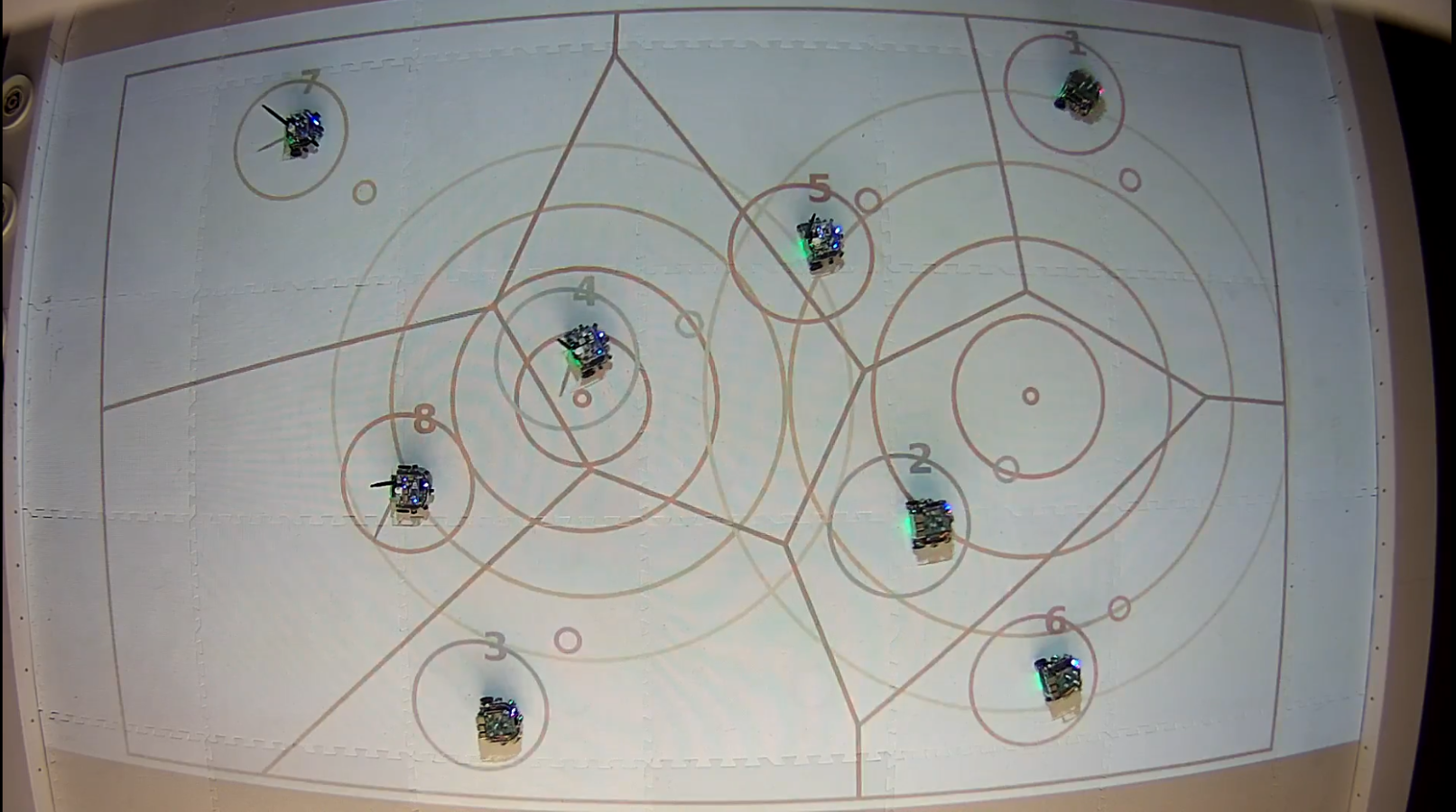}}
  \subfigure[]{\includegraphics[width=.475\linewidth]{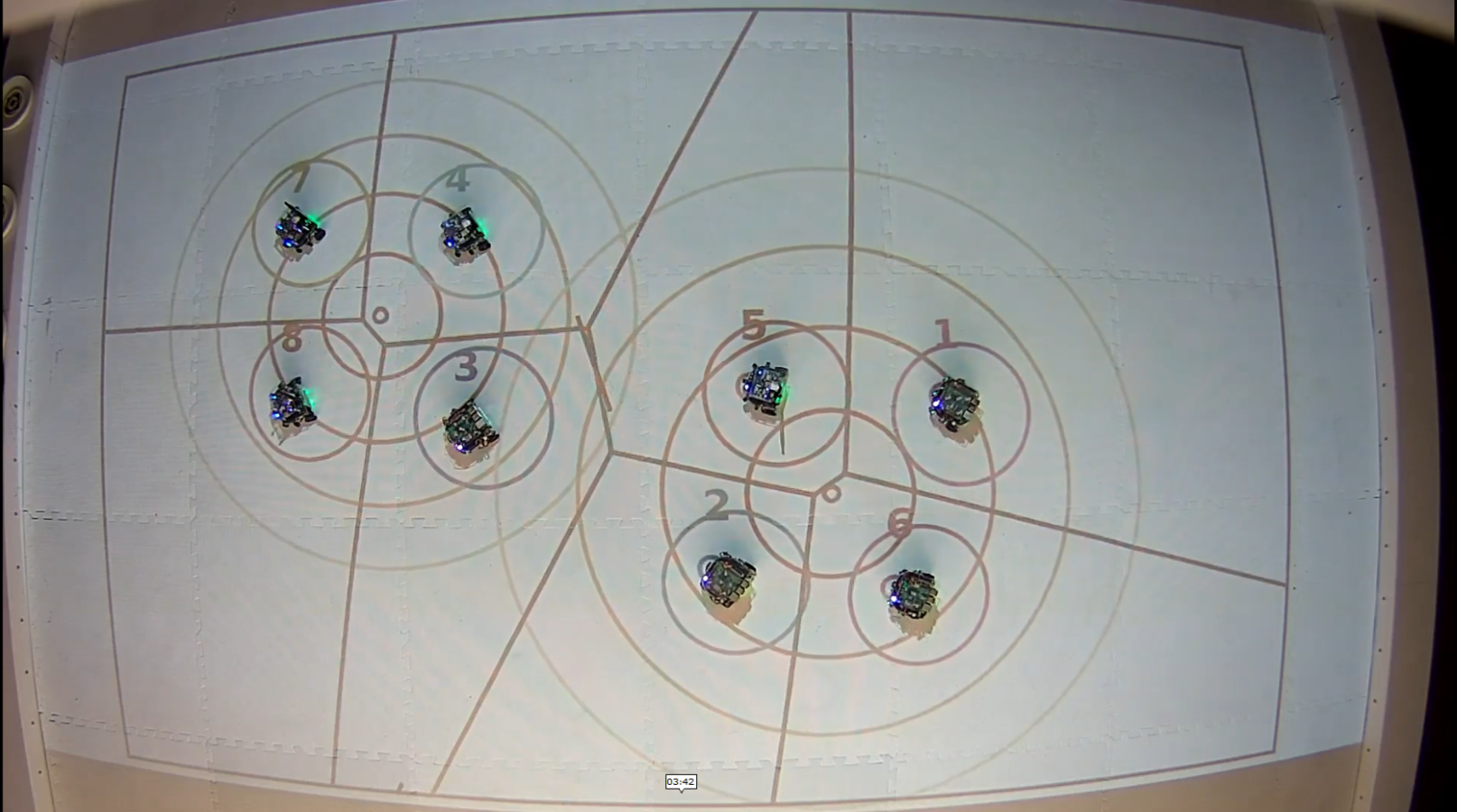}}
  \\
  \subfigure[]{\includegraphics[width=.475\linewidth]{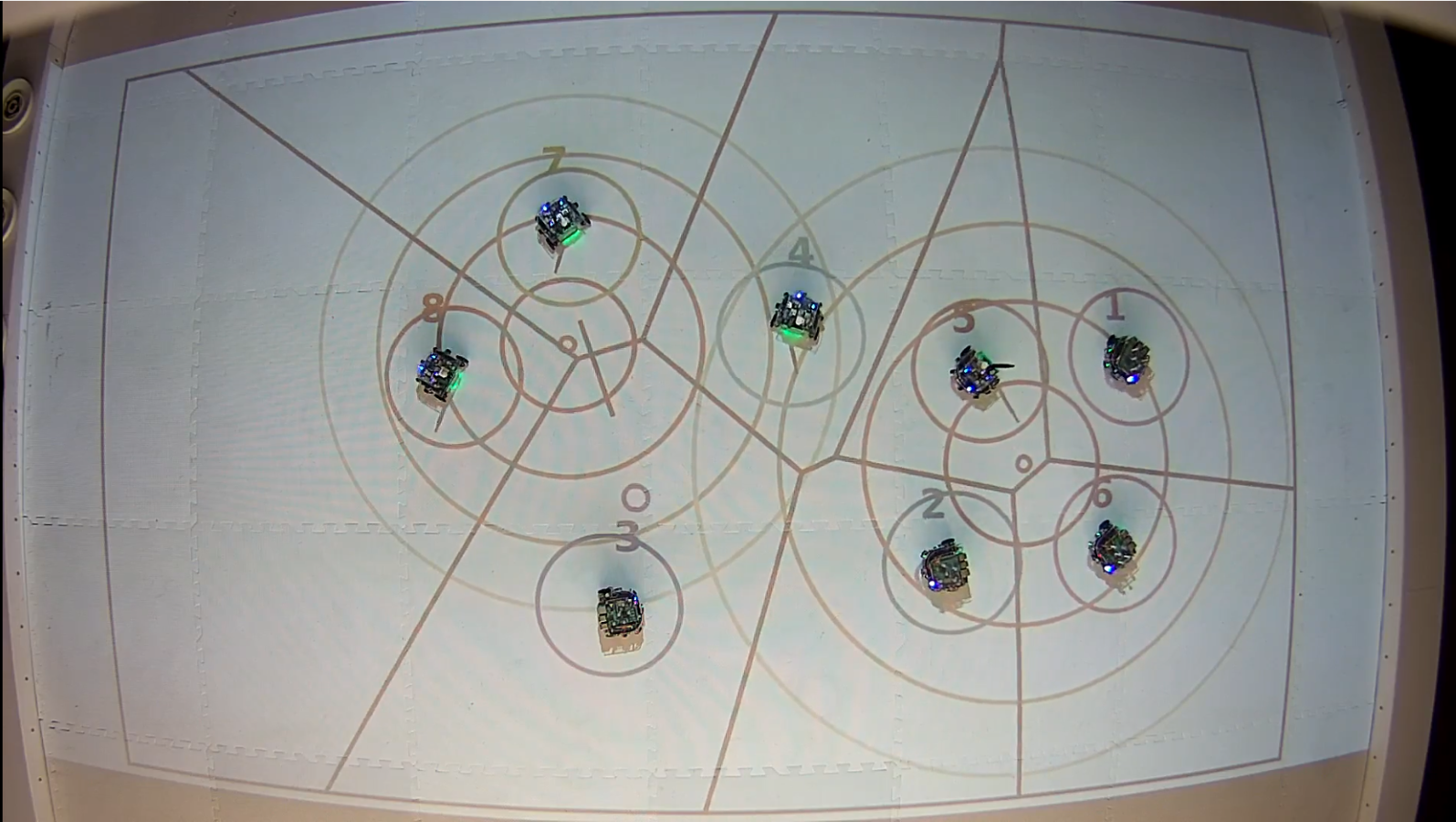}}
  \subfigure[]{\includegraphics[width=.475\linewidth]{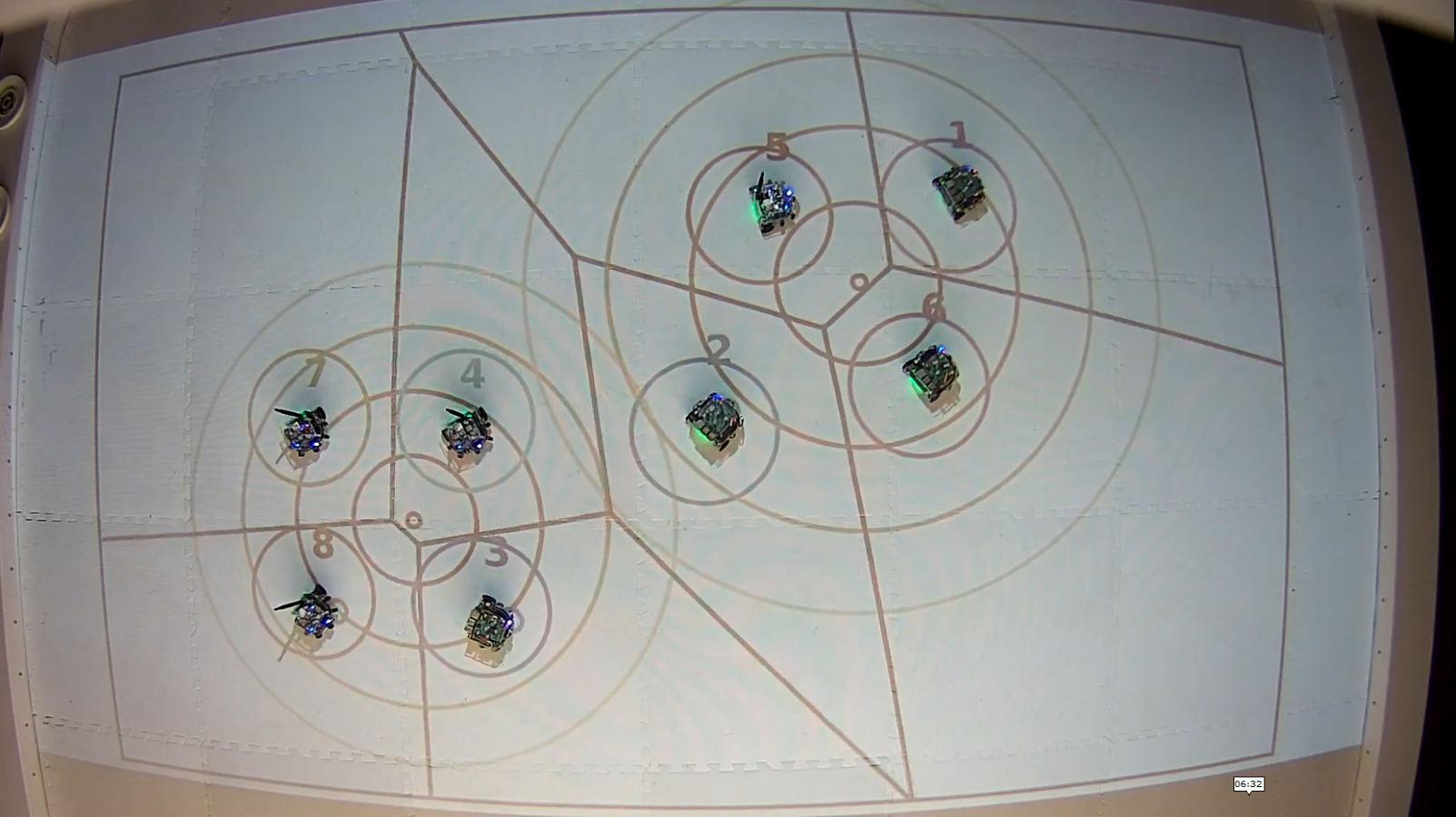}}
  \caption{Hardware experiments with $n=8$ differential drive robots  performing
    time-varying coverage control in the Robotarium, a remotely
    accessible swarm robotics testbed~\cite{SW-PG:20}.
    The robots execute the \tvd{sp}{\epsilon}
  coordination algorithm with the $2\setminus \! 1$-neighbors-delayed
  update~\eqref{eq:discrete-u2}.
 (a)  initialization, (b)-(d)  snapshots during the execution.
}\label{fig:robotarium}
\vspace*{-2ex}
\end{figure}

\section{Conclusions}\label{sec:conclusions}
We have presented a novel class of multi-agent coordination algorithms
for dynamic coverage control with time-varying density functions.  Our
approach is based on singular perturbation theory, specifically on
having agent position and control decision updates run at two
different timescales. The ratio between the (slow) timescale for agent
motion and the (fast) timescale for control updates is the
perturbation parameter. The inverse of this parameter has the natural
interpretation of the relative speed of the communication service with
respect to physical motion.  The resulting algorithm can be
implemented either in continuous or discrete time and is 2-hop
distributed over the Delaunay graph.  We have shown that, for small
enough values of the perturbation parameter, it achieves the same
performance as its centralized counterpart, i.e., tracking of a
centroidal Voronoi tessellation.  We have also proposed discrete-time
versions that rely only on 1-hop communication at the cost of having
to use delayed information, resulting in the all-neighbors-,
$2\setminus \! 1$-neighbors-, and $2$-neighbors-delayed variants, and
have formally established their asymptotic convergence guarantees.

\appendix

\begin{lemma}\label{le:roots-matrix}
  Consider the matrices 
  $B= I - \delta \eta \overline{A}^f$ and $\Lambda = \delta \eta \overline{A}^d$,
  where $\overline{A}^f \succ 0$, and both
  $\overline{A}^f, \overline{A}^d$ are symmetric. Then
  $B^2 - 2 \Lambda \succ 0$ for $\delta \eta$ sufficiently small.
  \end{lemma}
  \begin{proof}
    Note that
    $ B^2 - 2 \Lambda = I - 2 \delta \eta \overline{A}^f + (\delta
    \eta)^2 (\overline{A}^{f})^2- 2\delta\eta \overline{A}^d$ is
    symmetric and has real eigenvalues.  By the Courant-Fischer's
    theorem, it holds that
    \begin{align*}
      &\subscr{\lambda}{min}(B^2 - 2\Lambda)
        = \min_{x \cdot x=1} x^\top (B^2 - 2 \Lambda ) x \\
      & = \min_{x \cdot x=1}\{
        1 - 2\delta \eta x^\top \overline{A}^f x + (\delta \eta)^2 x^\top
        (\overline{A}^{f})^2x - 2\delta \eta x^\top \overline{A}^dx\} 
      \\
      & \ge 1 - 2 \delta \eta \subscr{\lambda}{max}(\overline{A}^f) +
        (\delta  \eta)^2 \subscr{\lambda}{min}((\overline{A}^{f})^2) -
        2 \delta  \eta \subscr{\lambda}{max}(\overline{A}^d),
      \end{align*}
      as all involved matrices are symmetric.  Thus, if $\delta  \eta$ is
      sufficiently small, then $B-2\Gamma \succ 0$.
    \end{proof}

  \end{document}